\documentclass[11pt]{amsart} 
\usepackage[bottom=1in]{geometry}
\usepackage{esint} 
\usepackage{graphicx, amssymb, hyperref }
\usepackage[mathscr]{euscript}
\usepackage{todonotes}
\usepackage[english]{babel}
  \usepackage{enumitem}
\newtheorem{theorem}{Theorem}
\newtheorem{proposition}[theorem]{Proposition}
\newtheorem{definition}[theorem]{Definition}

\newtheorem{corollary}[theorem]{Corollary}

\newtheorem{remark}[theorem]{Remark}

\DeclareMathOperator*{\esssup}{ess\,sup}
\newtheorem{thm}{Theorem} 

\newtheorem{conjecture}{Conjecture}

\newcommand{\one}{\mathbf{1}}

\newcommand{\rr}{\mathbb}

\newcommand{\ic}{\mathcal}
\newcommand{\ci}{\tilde{\chi}}

\newcommand{\ds}{\displaystyle}

\newtheorem{question*}{Question}
\newtheorem{conjecture*}{Conjecture}
\newtheorem*{main*}{\underline{Induction statement}}

\newcommand{\lft}{\left|}
\newcommand{\rg}{\right|}

\newcommand{\calD}{\mathcal{D}}
\newcommand{\calS}{\mathcal{S}}
\newcommand{\ch}{{\mathrm ch}}

\def\Xint#1{\mathchoice
   {\XXint\displaystyle\textstyle{#1}}%
   {\XXint\textstyle\scriptstyle{#1}}%
   {\XXint\scriptstyle\scriptscriptstyle{#1}}%
   {\XXint\scriptscriptstyle\scriptscriptstyle{#1}}%
   \!\int}
\def\XXint#1#2#3{{\setbox0=\hbox{$#1{#2#3}{\int}$}
     \vcenter{\hbox{$#2#3$}}\kern-.5\wd0}}

\def\aver#1{\Xint-_{#1}}

\title[Sparse bilinear forms for Bochner-Riesz multipliers]{Sparse bilinear forms for Bochner Riesz multipliers and applications}

\author{Cristina Benea}
\address{Cristina Benea, CNRS - Universit\'{e} de Nantes, Laboratoire Jean Leray, Nantes 44322, France}
\email{cristina.benea@univ-nantes.fr}

\author{Fr\'{e}d\'{e}ric Bernicot}
\address{Fr\'{e}d\'{e}ric Bernicot, CNRS - Universit\'{e} de Nantes, Laboratoire Jean Leray, Nantes 44322, France}
\email{frederic.bernicot@univ-nantes.fr}

\author{Teresa Luque}
\address{Teresa Luque, Instituto de Ciencias Matem\'aticas CSIC-UAM-UC3M-UCM, C/ Nicol\'as Ca\-bre\-ra, 13-15, 28049 Madrid, Spain} 
\email{teresa.luque@icmat.es}

\date{\today}


\thanks{The authors are supported by ERC project FAnFArE no. $637510$}

\begin{document}

\begin{abstract} We use the very recent approach developed by Lacey in \cite{LaceyA_2} and extended by Bernicot-Frey-Petermichl in \cite{weights_beyond_CZ}, in order to control Bochner-Riesz operators by a sparse bilinear form. In this way, new quantitative weighted estimates, as well as vector-valued inequalities are deduced. 
\end{abstract}

\maketitle

\section{Introduction}

During the last ten years it has been of great interest to obtain optimal operator norm estimates in Lebesgue spaces endowed with Muckenhoupt weights. More precisely, one asks for the growth of the norm of certain operators, such as the Hilbert transform or the Hardy-Littlewood maximal function, with respect to a characteristic assigned to the weight. Originally, the main motivation for sharp estimates of this type came from various important applications to partial differential equations. See for example  Fefferman-Kenig-Pipher \cite{FKP}, Astala-Iwaniec-Saksman \cite{AIS}, Petermichl-Volberg \cite{PetermichlVolberg}. The optimal result for Calder\'on-Zygmund operators (C-Z operators), which corresponds to the so-called {\it $A_2$ conjecture}, was first obtained by Hyt\"onen \cite{Hytonen}. 

Following the initial proof of the $A_2$ theorem, that involved the control of C-Z operators by dyadic operators, several other proofs were presented, each of them simplifying the proof and contributing to a better understanding of the field. Firstly, Lerner \cite{Lerner2, Lerner3} proved that the norm of a C-Z operator in a Banach function space is dominated by \emph {sparse} positive operators. After that, Lerner-Nazarov \cite{LernerNazarov} and Conde-Rey \cite{Conde-Rey} simultaneously simplified the norm control by a pointwise bound. Very recently, Lacey \cite{LaceyA_2} introduced a new method of establishing such pointwise control  by bringing into play the maximal truncations of the C-Z operator.

It turns out that among the family of dyadic intervals the operator is acting on, there is a certain \emph{sparse} (in the sense of \eqref{eq:sparse}) subcollection that dictates the overall behavior. Of course, this collection depends on the input of the operator, and it can be constructed in several ways. In \cite{LaceyA_2}, the sparse family corresponds to maximal coverings of the level sets of the maximal truncation of the C-Z operator. Then the sparsity property allows one to insert the weights and recover the best power for the $A_p$ constant.

The approach developed by Lacey was extended by Bernicot-Frey-Petermichl in \cite{weights_beyond_CZ} for non-integral singular operators, such as the Riesz transform associated with the heat semigroup. In this situation, the full range of Lebesgue exponents $p$ for which these operators are bounded is a sub-range $(p_0,q_0) \subset (1,\infty)$. The theory of such operators was developed by Auscher and Martell in \cite{AuscherMartell}.

Beyond C-Z theory, pointwise estimates do not always hold and only some integral estimates are possible. To deal with these kind of operators, \cite{weights_beyond_CZ} extends the notion of sparse operators to sparse bilinear forms, deducing from here sharp norm weighted estimates for such operators. The goal of this paper is to combine the new techniques from \cite{LaceyA_2} and \cite{weights_beyond_CZ}, in order to control Bochner-Riesz multipliers by sparse bilinear forms. Once we have obtained such a result, we can immediately infer quantitative weighted norm estimates for the Bochner-Riesz operator.

In the case of Bochner-Riesz multipliers, pointwise control by sparse operators is out of question because of their range of boundedness (see \eqref{def-cond-delta}). Moreover, we emphasize that the results in \cite{weights_beyond_CZ} cannot be directly adapted to the Bochner-Riesz case since the kernel of the operators studied there have at least some decaying properties (not pointwisely but at least in some $L^p$-sense) that fail in our case. Although the kernel of the Bochner-Riesz operator is smooth, the main difficulty relies on its lack of decay at infinity. This corresponds in frequency to the lack of smoothness of the symbol, in spite of it being compactly supported.

\medskip
This article is organized as follows: in Section \ref{sec:defResult} we define the Bochner-Riesz operator and state our main results concerning the bilinear form estimate (see Theorem \ref{thm:sparse_coll}) and its applications (see Theorem \ref{thm:weightsBR-general} and Corollary \ref{cor:vector-valued}). For clarity in the exposition, in this section, we state the results in dimension two. Section \ref{sec:Control}, \ref{sec:prop}  and \ref{sec:maximal_op} are devoted to the proof of Theorem \ref{thm:sparse_coll}. Section \ref{sec:dim2} details the two dimensional case where the Bochner- Riesz conjecture is solved. Finally, Section \ref{sec:weights} presents the quantitative weighted norm inequalities (Theorem \ref{thm:weightsBR-general}) and vector-valued extensions (Corollary \ref{cor:vector-valued}) we have obtained.

\section{Definitions and Main results}\label{sec:defResult}
In order to state our main results for the sparse bilinear form that we deal with and the weighted norm estimates we obtain, we recall the definition and properties of the Bochner-Riesz multipliers.

\medskip
\paragraph{\bf Bochner-Riesz operators} In ${\mathbb R}^n$ the Bochner-Riesz operator $\ic{B}^\delta$, for $\delta \geq 0$, is the linear Fourier multiplier associated with the symbol $(1-|\xi|^2)^{\delta}_+$, where $t_+=\max(t,0)$; that is,  the Bochner-Riesz operator is defined, on the class $\ic{S}\left( \rr{R}^n \right)$ of Schwartz function, by
\begin{equation*}
 \ic{B}^\delta(f) (x):= \int_{\rr{R}^n} e^{2\pi i x\xi} (1-|\xi|^2)_+^{\delta} \widehat{f}(\xi )\, d\xi.
 \end{equation*}
Since the symbol is in $L^\infty$, $\ic{B}^\delta$ is easily bounded on $L^2({\rr R}^n)$, for every $\delta \geq 0$. The case $\delta=0$ corresponds to the so-called {\it ball multiplier}, which is known to be unbounded on $L^p(\rr{R}^n)$ if $n \geq 2$ and $p\neq 2$. This is the celebrated result of C. Fefferman from \cite{ball_multiplier}. 

The main feature of the symbol is the singularity which is supported on the whole unit sphere. The Bochner-Riesz conjecture aims to describe what extra regularity (in terms of $\delta>0$) is sufficient to make $\ic{B}^\delta$ bounded in $L^p$. More precisely, the conjecture is the following: for $p\in (1,\infty) \setminus \{2\}$ then $\ic{B}^\delta$ is $L^p$-bounded if and only if 
\begin{equation}
\label{def-cond-delta}
\delta > \delta(p):=\max\left\{n\left|\frac{1}{p} -\frac{1}{2}\right| - \frac{1}{2} , 0 \right\}.
\end{equation}
Herz \cite{Herz} proved that the condition \eqref{def-cond-delta} on $\delta$ is necessary for the $L^p$-boundedness. It is not known whether this is also sufficient, except in two dimensions, where the conjecture was completely solved  by Carleson and Sj\"olin \cite{Carleson-Sjolin}.

In dimension greater than two, the full conjecture is still open and numerous works aim to contribute to this question. We just refer the reader to \cite{Tao_restriction-conjecture}, \cite{Stein_problems}, \cite[Chapter 8, Sections 5 and 8.3]{Duo-book} \cite[Chapter 10]{GrafakosMF}, \cite[Chapter 2, Chapter 3]{LuYan} and references therein for more details about this conjecture and how it is related to other problems (linear - bilinear restriction inequalities, ...). We just recall that for $\delta> (n-1)/2$, $\ic{B}^\delta$ is bounded pointwise by the Hardy-Littlewood maximal operator, and so the study of $\ic{B}^\delta$ fits into the classical theory. But for $0<\delta<(n-1)/2$, $\ic{B}^\delta$ is relatively difficult to study because its symbol is singular along a hypersurface (the sphere) with a non-vanishing curvature, which implies that the kernel (in space) has no good decay at infinity.

\medskip

\paragraph{\bf{Sparse families and bilinear form estimates.}} For a cube  $P\in\mathbb{R}^n$ we denote by $\calD(P)$ the mesh of dyadic cubes associated to $P$. The dyadic children of $P$ are produced by dividing each side of $P$ into two equal parts, while the dyadic parent of $P$ is the cube whose sidelengths are double the corresponding sidelengths of $P$, and which contains $P$. Thus every $P\in\mathbb{R}^n$ has exactly $2^n$ dyadic children and is contained in a unique dyadic parent. 

A collection of dyadic cubes $\calS:=(P)_{P\in \calS}$ is said to be {\it sparse} if for each $P\in \calS$ one has
\begin{equation}
\sum_{Q\in \ch_{\calS}(P)} |Q| \leq \frac{1}{2} |P|, \label{eq:sparse} \end{equation}
where $\ch_{\calS}(P)$ is the collection of $\calS$-children of $P$; that is, the maximal elements of $\calS$ that are strictly contained in $P$. 

Using the above definitions, we can now formulate our main result, which we present first in a simpler form:
\begin{theorem}
\label{thm:sparse-2dim-simple}
 In ${\rr R}^2$, consider $\delta>\frac{1}{6}=\delta(\frac{6}{5})$. For any compactly supported functions $f \in L^{\frac{6}{5}}$ and $g \in L^2$, there exists a sparse collection $\calS$ (depending on $f, g$) with 
\begin{align*}
	\left| \langle \ic{B}^\delta f, g  \rangle\right| 
	\leq C \sum_{Q \in \calS}  \left(\aver{6Q} |f|^{\frac{6}{5}}\, dx\right)^{\frac{5}{6}} \left(\aver{6Q} |g|^{2}\, dx\right)^{1/2} |Q| .
\end{align*}
\end{theorem}

In fact, we can obtain a more general result in dimension $n \geq 2$, which is formulated in Section \ref{sec:Control} Theorem \ref{thm:sparse_coll}. We emphasize that this is the first result which describes a control of Bochner-Riesz multipliers via \emph{sparse} operators. However, it was long known that the Bochner-Riesz multipliers can be split into \emph{``local"} operators, at different scales, a feature that we are also using implicitly.

We will now detail some applications for weighted estimates for the Bochner-Riesz operator,  as well as certain vector-valued extensions.

\medskip

\paragraph{\bf Weighted norm inequalities and vector-valued estimates} Weighted norm inequalities for the Bochner-Riesz operators and their corresponding maximal versions have been studied in several different contexts, and there are still many open questions. In the present work we are interested in weighted estimates involving Muckenhoupt weights and a few results are known in this sense. For other weighted estimates existing in the literature we refer for instance to Cordoba \cite{CorDisc}, Carbery \cite{CarMax}, Carbery-Seeger \cite{CarSeeger} where the Fefferman-Stein type weighted estimates have been thoroughly studied, or to Ciaurri-Stempak-Varona \cite{CiVaStem} for more general two-weights norm inequalities.

The $A_p$ theory for the Bochner-Riesz operator matches perfectly whenever $\delta > (n-1)/2$.  In this case, for every $p\in(1,\infty)$ and $w\in A_p$
\begin{equation*}
\|\ic B^\delta\|_{L^p(w)\to L^p(w)}\leq [w]_{A_p}^{\frac{1}{p-1}}.
\end{equation*}
This is a consequence of Buckley's sharp estimate for the Hardy-Littlewood maximal operator (see \cite{Buckley}).

For the critical index $\delta=(n-1)/2$, Shi and Sun \cite{ShiSun} proved  that $\ic B^{(n-1)/2}$ also maps $L^p(w)$ into $L^p(w)$  for every $p\in(1,\infty)$ and $w\in A_p$ (see the alternative proof of Duoandikoetxea and Rubio de Francia \cite{JaviRdF}). Moreover, in this particular case, Vargas  \cite{Vargas} showed that $\ic{B}^{\delta}$ is of weak type $(1,1)$ with respect to $A_1$ weights. However, the optimal quantitative estimate, i.e. the sharp dependance of the norm with respect to the characteristic of the weight, remains unknown in the critical index case. See \cite{CorrigeLiSun} and \cite{LPR} for further discussion in this sense.

Below the critical index,  since $\ic{B}^\delta$ is not bounded on the whole range $(1,\infty)$, we cannot expect not even one single $p_0$ so that $\ic{B}^\delta: L^{p_0}(w) \to L^{p_0}(w)$ for every weight $w \in A_{p_0}$. As the work \cite{AuscherMartell} points out, it is natural in this kind of situation to look for weights in a subclass of $A_p$, obtained as an intersection between a Muckenhoupt class and a Reverse H\"older class. In fact, as proved in \cite{Ap-RH}, this subfamily of weights corresponds to fractional powers of $A_p$ weights. See Section \ref{sec:weights} Proposition \ref{prop:Ap-Rh} for the precise equivalence. To our knowledge, the results concerning $A_p$ weights in this range of $\delta$ are the following:

\begin{thm}\label{th:knownW} Let $n\geq 2$, then:
\begin{enumerate}[label=(\roman*), ref=\roman*]
 \item {\rm Christ \cite{Christ-weightedBR-aeConv}:} 
\label{item-christ} If $(n-1)/2(n+1)< \delta < (n-1)/2$, then $\ic{B}^\delta$ is bounded on $L^2 (w)$ whenever $w^{n/(1+2\delta)}\in A_1$.
 \item {\rm Carro-Duoandikoetxea-Lorente \cite{CDL}:}
 \label{item-CarroDuo} If $0< \delta <(n-1)/2$, then $\ic{B}^\delta$ is bounded on $L^2 (w)$ for every weight $w$ such that $w^{(n-1)/2\delta}\in A_2$. 
\item{\rm Duoandikoetxea-Moyua-Oruetxebarria-Seijo \cite{DMOS}:} If $0 < \delta < \frac{n-1}{2}$, then $\ic{B}^\delta$ is bounded on $L^2 (w)$ for every radial weight $w$ such that $w^{n/(1+2\delta)} \in A_2$.
\end{enumerate}
\end{thm}

We remark in here that the results in papers \cite{Christ-weightedBR-aeConv} and \cite{CDL} are formulated for the maximal Bochner Riesz operator and therefore they are not expected to be optimal. Moreover, we note that the statement in \eqref{item-CarroDuo} can be obtained through complex interpolation between the $L^2$ weighted estimate for the critical index $\delta=(n-1)/2$ and the unweighted $L^2$  boundedness of $\ic{B}^0$. 

The best weighted result one can expect for $\ic{B}^\delta$, in terms of powers of $A_p$, is \eqref{item-christ}, with $A_1$ replaced by $A_2$:
\begin{equation}
\label{eq:optimal-res-weights}
\ic{B}^\delta : L^2(w) \to L^2(w), \quad \text{   for all $w$ such that    } w^{n/(1+2\delta)}\in A_2.
\end{equation}
Observe that if true, \eqref{eq:optimal-res-weights} would imply the full Bochner-Riesz conjecture, through extrapolation.

We obtain quantitative weighted estimates for $\ic{B}^\delta$ in a wide range of exponents $p$, that is not necessarily optimal. These weighted estimates follow from the control of $\ic{B}^\delta$ by a sparse bilinear form, as described in Theorem \ref{thm:sparse-2dim-simple} (dimension $n=2$) and its general version (Theorem \ref{thm:sparse_coll}, Section \ref{sec:Control}). We first indicate the simpler two-dimensional case, for $p\in(6/5,2)$. The general result for higher dimensions can be found in Theorem \ref{thm:weightsBR-general} (cf Section \ref{sec:weights}).

\begin{theorem}\label{thm:weightsBR-simple-2dim}
Let $\delta> 1/6$ and $6/5 \leq p <2 $. Then for all weights $w \in A_{\frac{5p}{6}} \cap RH_{\left( \frac{2}{p}  \right)'}$ we have
\begin{equation}\label{eq:WQuantdim2}
\|\ic{B}^\delta\|_{L^p(w) \to L^p(w)}\leq\left([w]_{A_{\frac{5p}{6}}}[w]_{RH_{(\frac{2}{p})'}}\right)^{\alpha}
\end{equation}
with $\alpha:=max\{1/(p-6/5),1/(2-p)\}$.
\end{theorem}

We emphasize that \eqref{eq:WQuantdim2} is not expected to be sharp; see Section \ref{sec:weights}, Remark \ref{re:compareWtQuan} for further discussions on this aspect.

Moreover, as explained later in Section \ref{sec:dim2}, there is some room to play with the exponents. In particular, this allows us to get some new vector-valued estimates for Bochner-Riesz operators (see  Corollary \ref{cor:vector-valued}).


\section{Control of the Bochner-Riesz multipliers by a sparse bilinear form}\label{sec:Control}
In this section we prove the general version of Theorem \ref{thm:sparse-2dim-simple}. Aiming that, we introduce a new critical index $\bar{\delta}_n$.

Let $p_0\in(1,2)$. Consider $p_1$ given by
\begin{equation}
\label{def:p_1}
 p_1:= 2\left(\frac{3}{2}-\frac{1}{p_0} \right) \in (1,2),
 \end{equation}
and let $\tilde \delta(p_1)$ be the smallest positive exponent such that 
$$ \ic{B}^{\tilde \delta(p_1)} : L^{p_1} \to L^{p_1}.$$
We know that $\tilde \delta(p_1) \geq\delta(p_1)$, and their equality is equvalent to the Bochner-Riesz conjecture.

Then $\bar{\delta}_n$ is defined as 
\begin{equation}
\label{eq:delta_n-bar}
\bar{\delta}_n(p_0):=\tilde \delta(p_1) + \frac{n-1}{2} \left(\frac{1}{p_0}-\frac{1}{2}\right).
\end{equation}


\begin{theorem}\label{thm:sparse_coll}
For any $1<p_0<2$, and any $\delta > \bar{\delta}_n(p_0)$ given by \eqref{eq:delta_n-bar}, and all functions $f, g$ compactly supported, there exists a sparse collection $\calS$ (depending on $f,g$) with 
\begin{align}\label{eq:bilin}
	\left| \langle\ic{B}^\delta(f), g \rangle \right| 
	\leq C \sum_{Q \in \calS}  \left(\aver{6Q} |f|^{p_0}\, dx\right)^{1/p_0} \left(\aver{6Q} |g|^{2}\, dx\right)^{1/2} |Q| .
\end{align}
\end{theorem}

\begin{remark}\label{re:p0delta}
Alternatively, the above theorem can be reformulated, so that, for any $\delta>0$, there exists a critical exponent $p_0(\delta)\in(1,2)$ for which the sparse form becomes
\begin{align*}
	\left| \int  \ic  B^\delta(f) \cdot g\, dx \right| 
	\leq C \sum_{Q \in \calS}  \left(\aver{6Q} |f|^{\tilde{p}}\, dx\right)^{1/\tilde{p}} \left(\aver{6Q} |g|^{2}\, dx\right)^{1/2} |Q|,
\end{align*}
where $\tilde{p}\in (p_0(\delta),2)$.
\end{remark}

The proof for Theorem \ref{thm:sparse_coll} follows ideas from \cite{LaceyA_2} and \cite{weights_beyond_CZ}. We first sketch the steps of the proof and then we introduce the definitions of the new maximal operators that are necessary to complete the proof. We refer the reader to \cite{weights_beyond_CZ}, where the whole strategy is presented in details, in the case of a smooth Fourier multiplier.

\medskip
\begin{proof}[Sketch of the proof for Theorem \ref{thm:sparse_coll}]
Assuming $f$ and $g$ are supported inside $6Q_0$, we define an exceptional set in the following way:
\begin{equation*}
E:=\left\lbrace x \in Q_0: \ic{B}^{\delta,*}(f)(x)+ \ic{B}^{\delta,**}(f)(x) +\ic{M}_{p_0}\left[ f \right](x)> C \left( \aver{6 Q_0} \lft f \rg^{p_0} \right)^{1/{p_0}} \right\rbrace,
\end{equation*}
for some large enough numerical constant $C$.
The operators $\ic{B}^{\delta,*}$ and $\ic{B}^{\delta,**}$ are certain maximal operators which will be defined later (and will be proved to be of weak type $(p_0,p_0)$), and $\ic{M}_{p_0}$ is the $L^{p_0}$- Hardy-Littlewood maximal function. Then $E$ is a proper open subset or ${\rr R}^n$ and so we may consider $(Q_j)_j$ to be a covering of $E$ by maximal dyadic cubes contained inside $E$: $E=\bigcup_{j} Q_j$. Using the weak-type $(p_0,p_0)$ boundedness of $\ic{B}^{\delta, *}, \ic{B}^{\delta, **}$ and $\ic{M}_{p_0}$, we also know that
$$ \lft E \rg=\sum_j |Q_j| \lesssim C^{-1} |Q_0|.$$
Maximality implies that $\widetilde{Q_j}$, the dyadic parent of $Q_j$, is not entirely contained in $E$. Moreover, for $C$ large enough, $\widetilde{Q_j} \subset Q_0$ and we can conclude that $\ic{B}^{\delta, *}, \ic{B}^{\delta, **}$ and $\ic{M}_{p_0}$ are going to be small at certain points of $\widetilde{Q_j}$. 

The sparse family of cubes from Theorem \ref{thm:sparse_coll} will be constructed iteratively, as in \cite{LaceyA_2}. We initialize $\ds \ic{Q}_{sparse}:=\lbrace Q_0 \rbrace$ and we will see that $\ds \ic{B}^\delta(f) \cdot \one_{Q_0 \setminus E}$ is under control, and hence we only need to deal with $\ic{B}^\delta(f) \cdot \one_{Q_j}$ for every $Q_j$ as above. We note that 
\begin{equation}
\label{eq:BR-split}
 \ic{B}^\delta(f)  \cdot \one_{Q_j}= \underbrace{\ic{B}^\delta \left(f \cdot \one_{6 Q_j} \right) \cdot \one_{Q_j}}_{\text{this goes in the iteration process}} + \ic{B}^\delta \left( f \cdot \one_{\left( 6 Q_j \right)^c} \right) \cdot \one_{Q_j}.
\end{equation}
The first term will be used in the inductive procedure, hence we update $\ds \ic{Q}_{sparse}=\ic{Q}_{sparse} \cup \lbrace Q_j: Q_j \subseteq E \text{  maximal dyadic interval}  \rbrace$. The second one is the ``off-diagonal term", for which we will prove that
\begin{equation}
\label{eq:off-diag}
\sum_{Q_j} \lft  \int_{Q_j} \ic{B}^\delta \left( f \cdot \one_{\left( 6 Q_j \right)^c} \right)(x)  \bar{g}(x) dx \rg \lesssim \left( \aver{6Q_0} \lft f\rg^{p_0} dx \right)^{1/{p_0}} \cdot \left( \aver{6Q_0} \lft g\rg^{2} dx \right)^{1/2} \cdot \lft Q_0 \rg.
\end{equation}
Once we have the above inequality, it remains only to check that the collection of cubes is indeed sparse. This becomes clear once we realize that these cubes constitute in fact maximal coverings for the level sets of the maximal operators $\ic{B}^{\delta, *}, \ic{B}^{\delta, **}$ and $\ic{M}_{p_0}$, with a constant $C$ which can be chosen large enough.

The difficult task is to prove \eqref{eq:off-diag}, since we need to define a suitable maximal operator $\ic{B}^{\delta, *}$ and prove its boundedness. 
\end{proof}

Let us first introduce some elementary operators, which correspond to an $L^\infty$-normalized Littlewood-Paley decomposition for the Bochner-Riesz multiplier. We start with a positive, smooth function $\chi$, which is constantly equal to $1$ on the interval $\ds \left[ \frac{1}{2}+\frac{1}{100}, 1-\frac{1}{100}\right]$, and is supported on $\ds\left[ \frac{1}{2}, 1+\frac{1}{100}\right]$. Moreover, we require $\chi$ to satisfy
\begin{equation*}
\sum_{k \leq 0} \chi(2^{-k}x)=1 \qquad \forall x,\ |x|\leq 1.
\end{equation*}

We then consider the $L^\infty$- normalized symbols
\begin{equation*}
s_k(\xi):=2^{-k \delta} \left( 1-\lft \xi \rg^2  \right)^{\delta}_+ \cdot \chi\left( 2^{-k} \left( 1- \lft \xi \rg^2\right) \right).
\end{equation*}
We denote $S_k$ the Fourier multiplier associated to the symbol $s_k$. This yields the following decomposition for the Bochner-Riesz operator:
\begin{equation*}
\ic{B}^\delta (f)(x)=\sum_{k \leq 0} 2^{k \delta} S_k(f)(x).
\end{equation*}

\begin{definition}[The Maximal Operators]
For any $\epsilon >0$, we define
\begin{equation*}
\ic{B}^\delta_\epsilon(f)(x):=\int_{\rr{R}^n} \hat{f}(\xi) \left( 1-\lft \xi \rg^2  \right)^\delta_+ \ci\Big( \epsilon \big( 1-\lft \xi \rg^2  \big)   \Big) e^{2 \pi i x \cdot \xi} \, d\xi,
\end{equation*}
where $\ci$ is a function supported on $\ds \left[ -\frac{1}{100}, 1+\frac{1}{100}\right]$, and constantly equal to $1$ on $\left[0, 1\right]$.
If $\epsilon$ is small enough, then $\ic{B}^\delta_\epsilon$ coincides with $\ic{B}^\delta$. Then we define the ``off-diagonal" maximal operator
\begin{equation*}
\ic{B}^{\delta, *}(f)(x):= \sup_{\epsilon>0} \ \sup_{|x-y|<\epsilon}\ \left( \aver{B\left(y, \epsilon\right)} \lft \ic{B}^\delta_\epsilon \left( f \cdot \one_{B\left( x, 3 \epsilon \right)^c}  \right)(z) \rg^2 dz  \right)^{1/2}.
\end{equation*}
We will also need an auxiliary maximal operator $\ic{B}^{\delta, **}$, which is defined as
\begin{equation*}
\ic{B}^{\delta, **}(f)(x):= \sup_{\epsilon>0} \ \sup_{|x-y|<\epsilon} \ \left( \aver{B\left(y, \epsilon\right)} \lft \ic{B}^\delta_\epsilon \left( f \right)(z) \rg^2 dz  \right)^{1/2}.
\end{equation*}
 We note that, once we know that $B^{\delta, **}$ is a bounded operator on some $L^p$ space, then by Lebesgue differentiation theorem, we have that
 $$   \lft \ic{B}^{\delta}(f)(x) \rg \leq \ds \ic{B}^{\delta,**}(f)(x) \qquad  \textrm{for a.e. $x$.}$$
\end{definition}

Our definitions of the maximal operators could seem unusual, but in fact they reflect a principle that was observed by Stein, presented in \cite{Feff-BR-restr}. More exactly, using the Restriction Theorem, it can be proved that, for $p_0 \leq 2$, the Bochner-Riesz operator $\ic{B}^\delta$ maps $L^{p_0}$ into $L^2$, locally. This will appear in our proof of the $L^{p_0} \mapsto L^{p_0, \infty}$ boundedness of $\ic{B}^{\delta,**}$ and $\ic{B}^{\delta, *}$.

The following Propositions (the proofs of which are postponed to the next sections) will play a crucial role for proving Theorem \ref{thm:sparse_coll}.

\begin{proposition}
\label{prop-Heisenberg}
Let $p_0 \in(1,2)$ and consider 
$$ \rho_n(p_0) := \left\{ \begin{array}{l}
 \frac{n-1}{2}\left(\frac{1}{p_0}-\frac{1}{2}\right), \textrm{  if $ 2\frac{n+1}{n+3}\leq p_0<2$} \\
 n\left(\frac{1}{p_0}-\frac{1}{2}\right) - \frac{1}{2},\textrm{  if $ 1<p_0\leq 2\frac{n+1}{n+3}$}.
\end{array} \right.$$
Then, for any ball $B_r$ of radius $r$ and $k \leq 0$ so that $2^k r\geq 1$, we have for an arbitrarily large exponent $M$
\begin{equation}
\label{eq:S_k_H}
\left(  \aver{B_r} \left| S_k\left(f \cdot \one_{\left(2B_r\right)^c}\right)(x) \right|^{2} dx \right)^{1/2} \leq C_{\rho, M} 2^{-k\rho} \sum_{j \geq 1} 2^{-jM} \left( \aver{2^{j+1} B_r \setminus 2^j B_r} \left| f(x)  \right|^{p_0} dx \right)^{1/p_0},
\end{equation}
whenever $\rho >\rho_n\left(p_0 \right)$.
\end{proposition}

\begin{proposition}
\label{non-Heisenberg}
Let $p_0 \in (1,2)$ and $\ds \rho_n\left(p_0\right)$ be as in Proposition \ref{prop-Heisenberg}. Then, for any ball of radius $\epsilon \geq 1$ and $k \leq 0$ so that $2^k \epsilon \leq 1$, and any $\ds \rho >\rho_n\left(p_0\right)$, we have
\begin{equation*}
\left( \aver{B(x, 2\epsilon)} \left| S_k\left( f \cdot \one_{B(x, 3 \epsilon)}  \right) \rg^2 dz  \right)^{1/2} \lesssim 2^{-k \rho} \left(  \aver{B(x, 3\epsilon)} \lft f \rg^{p_0} dz \right)^{1/{p_0}}.
\end{equation*}
\end{proposition}

\begin{proposition}
\label{propo:maxBR}
Then for every $\delta>\bar{\delta}_n(p_0)$ (given by \eqref{eq:delta_n-bar}) the maximal operators $\ic{B}^{\delta,*}$ and $\ic{B}^{\delta,**}$ are of weak-type $(p_0,p_0)$.
\end{proposition}

Now we are ready to show how the above propositions imply inequality \eqref{eq:off-diag}, which further completes the proof of Theorem \ref{thm:sparse_coll}. 

\begin{proof}[Proof of Theorem \ref{thm:sparse_coll}]
We have already given the sketch of the proof, it remains to check the estimate for the part associated to the integral over $Q_0\setminus E$ and the inequality \eqref{eq:off-diag}, which corresponds to the integral over each $Q_j$. \\

To begin with, is not difficult to check that
\begin{equation}
\label{eq:outside-exc}
\lft \int_{Q_0 \setminus E} \ic{B}^\delta (f)(x) \bar{g}(x) dx \rg \lesssim \left( \aver{6Q_0} \lft f \rg^{p_0} dx \right)^{1/{p_0}} \cdot \left( \aver{6Q_0} \lft g \rg^{2} dx \right)^{1/{2}} \cdot \lft Q_0 \rg.
\end{equation}
Indeed, we have already argued that $\ds \lft\ic{B}^{\delta}(f)(x) \rg \leq \ic{B}^{\delta,**}(f)(x)$ for a.e. $x$, and hence
\begin{align}
\label{eq:outside-exc-ptws}
&\lft \int_{Q_0 \setminus E} \ic{B}^\delta (f)(x) \bar{g}(x) dx \rg \lesssim \left(\sup_{ Q_0 \setminus E} \ic{B}^\delta (f)\right) \int_{Q_0} \lft g(x) \rg\, dx\\
& \lesssim \left(  \aver{6Q_0} \lft f \rg^{p_0}dx \right)^{1/{p_0}} \cdot \aver{6Q_0} \lft g(x) \rg dx \cdot \lft Q_0 \rg. \nonumber
\end{align}
Then H\"{o}lder's inequality implies \eqref{eq:outside-exc}.

It remains to prove  \eqref{eq:off-diag}.  Let $Q_j$ be a maximal dyadic cube as those appearing in \eqref{eq:BR-split}, and let $r$ denote its diameter. We further split $\ds \ic{B}^\delta(f \cdot \one_{\left(6 Q_j\right)^c})$ as 
\begin{equation*}
\ic{B}^\delta(f \cdot \one_{\left(6 Q_j\right)^c})(x)=\ic{B}_{2r}^\delta(f \cdot \one_{\left(6 Q_j\right)^c})+ \left(  \ic{B}^\delta(f \cdot \one_{\left(6 Q_j \right)^c})-\ic{B}_{2r}^\delta(f \cdot \one_{\left(6 Q_j\right)^c}) \right) .
\end{equation*}
To deal with the first term, note that by definition, for every $x\in Q_j$
$$\ds \lft \ic{B}_{2r}^\delta(f \cdot \one_{\left( 6Q_j \right)^c})(x) \rg \leq \inf_{y\in \widetilde{Q_j}} \ic{B}^{\delta,*}(f)(y) \lesssim C \left(  \aver{6Q_0} \lft f \rg^{p_0}dx \right)^{1/{p_0}}$$
since $Q_j$ is maximal and so $\widetilde{Q_j}$ (its parent) meets the complimentary set $E^c$. Actually, the definition of $\ic{B}^{\delta, *}$ is motivated by this inequality. Then we have an estimate similar to the one in \eqref{eq:outside-exc-ptws}:
\begin{align}
\label{eq:Q_j-estB_r}
\lft \int_{Q_j} \ic{B}^\delta (f \cdot \one_{\left(6 Q_j\right)^c})(x) \bar{g}(x) dx \rg & \lesssim \left(  \aver{6Q_0} \lft f \rg^{p_0}\, dx \right)^{1/{p_0}} \cdot \int_{Q_j} \lft g(x) \rg \, dx. \nonumber
\end{align}
Since the cubes $Q_j$ are disjoint and they are all contained inside $Q_0$ (because of the dyadic structure), we obtain that 
\[
\sum_{j} \int_{Q_j} \lft g(x) \rg dx \lesssim \int_{Q_0} \lft g(x)\rg dx \lesssim \left(  \aver{6 Q_0} \lft  g(x) \rg^2 dx \right)^{1/2} \cdot \lft  Q_0\rg,
\]
which yields
$$ \sum_{j} \lft \int_{Q_j} \ic{B}^\delta (f \cdot \one_{\left(6 Q_j\right)^c})(x) \bar{g}(x) \, dx \rg \lesssim \left(  \aver{6 Q_0} \lft f \rg^{p_0}\, dx \right)^{1/{p_0}} \left(  \aver{6 Q_0} \lft  g(x) \rg^2 dx \, \right)^{1/2} \cdot \lft  Q_0\rg,$$
as desired.

We are left with estimating $\ic{B}^\delta(f \cdot \one_{\left(6 Q_j \right)^c})-\ic{B}_{2r}^\delta(f \cdot \one_{\left(6 Q_j\right)^c})$, and in doing this, we use the operators $S_k$. More exactly, we note that
\begin{equation*}
\ic{B}^\delta(f \cdot \one_{\left(6 Q_j \right)^c})(x)-\ic{B}_{2r}^\delta(f \cdot \one_{\left(6 Q_j\right)^c})(x)=\sum_{2r \geq 2^{k+1}r \geq 1} 2^{k\delta} S_k\left( f \cdot \one_{\left( 6Q_j \right)^c}  \right)(x).
\end{equation*}
This implies that 
\begin{align*}
&\lft \int_{Q_j}\ic{B}^\delta(f \cdot \one_{\left(6 Q_j \right)^c})(x)-\ic{B}_{2r}^\delta(f \cdot \one_{\left(6 Q_j\right)^c})(x) \cdot \bar{g}(x) \, dx  \rg \\
&\lesssim \sum_{2r \geq 2^{k+1}r \geq 1} 2^{k \delta} \left( \frac{1}{\lft Q_j \rg} \int_{Q_j} \lft S_k\left( f \cdot \one_{\left( 6Q_j \right)^c}  \right)(x)\rg^2 dx \right)^{1/2} \cdot \left( \frac{1}{\lft Q_j \rg} \int_{Q_j} \lft g(x) \rg^2 dx  \right)^{1/2} \cdot \lft Q_j \rg.
\end{align*}
Proposition \ref{prop-Heisenberg} implies that the above expression can be bounded by
\begin{align*}
&\sum_{2r \geq 2^{k+1}r \geq 1} 2^{k\left( \delta -\rho\right)} \inf_{x \in 6Q_j} \ic{M}_{p_0}\left[f\right](x) \cdot \left( \frac{1}{\lft Q_j \rg} \int_{Q_j} \lft g(x) \rg^2 \, dx  \right)^{1/2} \cdot \lft Q_j \rg\\
&\lesssim \sum_{2r \geq 2^{k+1}r \geq 1} 2^{k\left( \delta -\rho\right)}  \left( \aver{6Q_0} \lft f \rg^{p_0} \, dx  \right)^{1/{p_0}} \cdot \left( \frac{1}{\lft Q_j \rg} \int_{Q_j} \lft g(x) \rg^2 \, dx  \right)^{1/2} \cdot \lft Q_j \rg,
\end{align*}
because $6Q_j$ meets the complementary $E^c$ (so the maximal function is bounded at some points of $6Q_j$).
Now we only need to notice that 
\[
\sum_{Q_j} \left( \frac{1}{\lft Q_j \rg} \int_{Q_j} \lft g(x) \rg^2 dx  \right)^{1/2} \cdot \lft Q_j \rg \lesssim \left( \frac{1}{\lft Q_0 \rg} \int_{Q_0} \lft g(x) \rg^2 dx  \right)^{1/2} \cdot \lft Q_0 \rg,
\]
which is a consequence of the disjointness of the cubes $Q_j$ and of the Cauchy-Schwartz inequality.

This completes the verification of \eqref{eq:off-diag}, which allows us to prove then Theorem \ref{thm:sparse_coll} by iteration, as explained before in the sketch.
\end{proof}

Consequently, we are left we proving Propositions \ref{prop-Heisenberg},  \ref{non-Heisenberg} and \ref{propo:maxBR}, and we will do this in the following sections.

\section{Localized estimates for the elementary operators $S_k$ (Propositions \ref{prop-Heisenberg} and \ref{non-Heisenberg}}\label{sec:prop}

The operators $S_k$ can be studied using oscillatory integrals, maximal Kakeya operators, or the Restriction Theorem. In our analysis of $\ic{B}^\delta$, we choose the latter strategy.

The $R\left( p \to q \right)$ Restriction Conjecture consists in obtaining the full range of exponents $(p,q) \in [1,\infty]^2$ such that
\begin{conjecture}
\begin{equation}
\label{eq:restriction_conj}
\tag{$R\left( p \to q \right)$} \left\| \hat{f} \vert_{{\mathbb S}^{n-1}}\right\|_{L^q\left( {\mathbb S}^{n-1} \right)} \leq C \left\| f \right\|_{L^p\left( \rr{R}^n \right)}
\end{equation}
holds for every smooth function $f \in \ic{S}\left(\rr{R}^n\right)$. 
\end{conjecture}

While the full conjecture is still open, the particular case $q=2$ was completely solved by Tomas and Stein \cite{Tomas:tomas-Stein}:
\begin{thm}
\label{thm:Tomas-Stein}
For any $n \geq 2$, the restriction inequality $R\left( p \to 2 \right)$ holds for any $1 \leq p \leq \frac{2n+2}{n+3}$. Moreover, this bound fails for any $p>\frac{2n+2}{n+3}$.
\end{thm}
That Theorem \ref{thm:Tomas-Stein} implies, in some cases, the boundedness of the Bochner-Riesz operator was known since \cite{Feff-BR-restr} and \cite{Christ-weightedBR-aeConv}. A deeper connection between the two problems was revealed in \cite{Tao-BR-impliesRestr}.

In this section, we will be using a result from \cite{Tao_weak-bounds}, where weak type-estimates for the critical $\delta$ are studied:
\begin{thm}[Theorem 1.1 of \cite{Tao_weak-bounds}]
\label{thm:tao-weak-endpoint}
For any $\ds p \leq \frac{2n+2}{n+3}$, the Bochner-Riesz operator $\ic{B}^{\delta_n\left( p\right)}$ satisfies the following:
\begin{equation*}
\int_{\lbrace \ic{M}_pf(x) \leq \lambda \rbrace}\left| \ic{B}^{\delta_n\left(p\right)} f(x) \right|^2 dx \leq C \lambda^{2-p} \left\|f\right\|_p^p,
\end{equation*}
where $\delta_n(p)$ is the critical index given by $\ds \delta_n(p):=n \left( \frac{1}{p}-\frac{1}{2} \right)-\frac{1}{2}$. Moreover, for any $k \leq 0$, we have
\begin{equation}
\label{eq:tao-S_k}
\int_{\lbrace \ic{M}_pf(x) \leq \lambda \rbrace}\left| S_k f(x) \right|^2 dx \leq C  2^{- 2k \delta_n(p)}\lambda^{2-p} \left\|f\right\|_p^p.
\end{equation}
\end{thm}

This last result uses in a crucial way the Restriction inequality $R\left( p \to 2 \right)$, and that is why the range for the exponent $p$ is the one given by Tomas-Stein Theorem \ref{thm:Tomas-Stein}. Using the above two results, we can now provide a proof for Propositions \ref{prop-Heisenberg} and \ref{non-Heisenberg}.

\begin{proof}[Proof of Proposition \ref{prop-Heisenberg}]
First, we  will split $f \cdot \one_{\left( 2B_r \right)^c}$ as
\begin{equation*}
f \cdot \one_{\left( 2B_r \right)^c}=\sum_{j \geq 1} f \cdot \one_{\left( 2^{j+1}B_r \setminus 2^j B_r \right)}:=\sum_{j \geq 1} f_j.
\end{equation*}
It will be enough to prove for every $j \geq 1$ that 
\begin{equation*}
\left(  \aver{B_r} \left| S_k\left(f _j\right)(x) \right|^{2} dx \right)^{2} \leq C_{\rho, M} 2^{-k\rho} 2^{-jM} \left( \aver{2^{j+1} B_r \setminus 2^j B_r} \left| f(x)  \right|^{p_0} dx \right)^{1/{p_0}},
\end{equation*}
where $M$ can be chosen to be as large as we wish.

The above estimate will be an application of Theorem \ref{thm:tao-weak-endpoint}. Let $1 \leq p \leq \frac{2n+2}{n+3}$, and note that, since $f_j$ is supported inside $2^{j+1}B_r \setminus 2^j B_r$, we have the inclusion
\begin{equation*}
B_r \subseteq \left\{ x : \ic{M}_p\left[ f_j \right](x) \leq \lambda:=\left( \aver{2^{j+1}B_r} \lft f_j\rg^{p} dz \right)^{1/p}  \right\}.
\end{equation*}
As a consequence of \eqref{eq:tao-S_k}, we have that 
\[
\int_{B_r} \lft S_k(f_j)(z) \rg^2 dz \leq C 2^{-2 k \delta_n(p)} \lambda^{2-p} \left\| f_j \right\|_p^p.
\]
This further leads to the estimate
\begin{equation}
\label{eq:av_2_S_k}
\left( \aver{B_r} \lft S_k(f_j)  \rg^2 dx \right)^{1/2} \lesssim 2^{-k \delta_n(p)} 2^{\frac{\left( j+1 \right) n}{2}} \left( \aver{2^{j+1}B_r} \lft f_j \rg^p dx  \right)^{1/p}.
\end{equation}
This is still far from the desired inequality \eqref{eq:S_k_H}. However, we can use the fast decay of $\ds \check{s}_k$. Indeed, by using that $s_k$ is smooth at the scale $2^k$ and supported in an annulus of measure $\simeq 2^k$, integration by parts yields
\begin{equation}
\label{eq:decay_s_k}
\lft  \check{s}_k(x)\rg \lesssim \frac{2^k}{ \left( 2^k \lft x\rg  \right)^N}, 
\end{equation}
for all $\lft x \rg$ sufficiently large and any integer $N$.

Using the fast decay of $\check{s}_k$, we have that, for any $x \in B_r$, 
\begin{equation*}
\lft S_k(f_j)(x) \rg=\lft \int_{\rr{R}^n} f_j(t) \check{s}_k(x-t) dt\rg \lesssim 2^k \left( 2^k 2^j r  \right)^{-N} \int_{2^{j+1}B_r} \lft f_j(t) \rg dt.
\end{equation*}
A straightforward computation yields that
\begin{equation}
\label{eq:S_k_decay}
\left( \aver{B_r} \lft S_k(f_j) \rg^2 dx  \right)^{1/2} \lesssim \| S_k(f_j)\|_{L^\infty\left(B_r\right)} \lesssim 2^k \left( 2^k 2^j r \right)^{-N} \left( 2^j r \right)^{n} \left( \aver{2^{j+1}B_r} \lft f_j \rg^p dx \right)^{1/p},
\end{equation}
where the implicit constants depend on $N, p$ and the dimension $n$. In the above identity, we have arbitrary decay in the form of $\left(2^j\right)^{-N}$, and this will compensate for the positive power of $2^j$ in \eqref{eq:av_2_S_k}. More exactly, we can interpolate inequalities \eqref{eq:av_2_S_k} and \eqref{eq:S_k_decay} to obtain
\begin{equation}
\label{eq:S_k_p_restr}
\left( \aver{B_r} \lft S_k(f_j) \rg^2 dx  \right)^{1/2} \lesssim 2^{-k \left[\delta_n(p) -\frac{1}{M}   \right]} 2^{-jM} \left( \aver{2^{j+1}B_r} \lft f_j \rg^p dx \right)^{1/p}.
\end{equation}
The equation above holds for any $1 \leq p <\frac{2n+2}{n+3}$. The last step is to interpolate between \eqref{eq:S_k_p_restr} and the trivial $L^2 \mapsto L^2$ estimate for $S_k$ in order to get, for any $p \leq p_0 \leq 2$,
\begin{equation*}
\left( \aver{B_r} \lft S_k(f_j) \rg^2 dx  \right)^{1/2} \lesssim 2^{-k \nu \left(\delta(p) -\frac{1}{M}   \right)} 2^{-j\tilde{M}} \left( \aver{2^{j+1}B_r} \lft f_j \rg^{p_0} dx \right)^{1/{p_0}}.
\end{equation*}
Here $\nu$ is the interpolation coefficient which gives $\dfrac{1}{p_0}=\dfrac{1-\nu}{2}+\dfrac{\nu}{p}$. Hence we obtain inequality \eqref{eq:S_k_H} for any $\ds \rho>\rho_n\left( p_0\right)$, where the critical $\ds \rho_n\left( p_0\right)$ is given by
\begin{equation*}
\rho_n\left( p_0\right):=\nu  \delta_n(p)=\left( n \left( \frac{1}{p}-\frac{1}{2} \right)-\frac{1}{2}\right) \cdot \frac{\frac{1}{p_0}-\frac{1}{2}}{\frac{1}{p}-\frac{1}{2}} .
\end{equation*}
There are two possibilities:
\begin{itemize}
\item[a)] if $\ds 1 \leq p_0 \leq \frac{2n+2}{n+3}$, then we have minimal $\rho_n(p_0)=\delta(p_0)=n \left( \dfrac{1}{p_0}-\dfrac{1}{2}\right)-\dfrac{1}{2}$ by taking $p=p_0$;
\item[b)] if $\ds \frac{2n+2}{n+3} \leq p_0 <2 $, then the minimal $\rho_n(p_0)$ is obtained for $p=\frac{2n+2}{n+3}$: $$\ds \rho_n(p_0)=\dfrac{n-1}{2}\left( \frac{1}{p_0}-\frac{1}{2}\right).$$
\end{itemize}

In other words,
\begin{equation}
\label{formula_rho_n}
\rho_n(p_0)=\max \left\lbrace  n \left( \dfrac{1}{p_0}-\dfrac{1}{2}\right)-\dfrac{1}{2}, \dfrac{n-1}{2}\left( \frac{1}{p_0}-\frac{1}{2}\right) \right \rbrace.
\end{equation}
\end{proof}

\begin{proof}[Proof of Proposition \ref{non-Heisenberg}]
We denote $\ds F:= f \cdot \one_{B(x, 3 \epsilon)}$, and use the following rough estimate:
\begin{align*}
&\int_{B(x, \epsilon)} \lft S_k (F)(z) \rg^2 dz \lesssim \left\| S_k(F)\right\|^2_2=\left\|  \widehat{S_k(F)} \right\|_2^2.
\end{align*}
We want to bound the $L^2$ norm of $\ds \widehat{S_k(F)}$ by the $L^{p}$ norm of $F$, for $p_0<2$. This is possible by making use of the Restriction Theorem in the following way:
\begin{align}
\label{eq:S_k_1}
&\left\| \widehat{S_k(F)}\right\|_2^2=\int_{\rr{R}^2} \lft s_k(\xi)\rg^2 \lft \widehat{F}(\xi) \rg^2 d\xi \\
&\lesssim \int_{1-2^{k+1}}^{1-2^{k-1}} \chi^2\left( 2^{-k}\left( 1-r^2  \right)  \right) \int_{{\mathbb S}^{n-1}} \lft \hat{F}(r \theta) \rg^2 d \theta r dr 
\lesssim 2^k \| F \|_{p}^2.\nonumber
\end{align}  

The inequality above is true whenever we have an $R\left(p \to 2 \right)$ restriction result, that is, whenever $\ds 1 \leq p \leq \dfrac{2\left(n+1  \right)}{n+3}$. Now we interpolate the inequality in \eqref{eq:S_k_1} with the much easier estimate $\ds \| S_k(F) \|_2 \lesssim \| F\|_2$, in order to obtain, for any $p \leq p_0 \leq 2$:
\begin{equation*}
\| S_k(F)\|_2 \lesssim 2^{\frac{k}{2} \cdot \theta} \|  F \|_{p_0}, \qquad \text{  where   } \theta=\frac{\frac{1}{p_0}-\frac{1}{2}}{\frac{1}{p}-\frac{1}{2}}.
\end{equation*}
To summarize, we have that 
\[
\left(  \int_{B(x, 2\epsilon)} \lft S_k(f \cdot \one_{B(x, 3 \epsilon)}) \rg^2 dz \right)^{1/2} \lesssim 2^{\frac{k}{2} \cdot \theta} \left( \int_{B(x, 3 \epsilon)} \lft f \rg^{p_0}  \right)^{1/{p_0}}.
\]
If we want to express it using averages, the above inequality becomes
\begin{equation*}
\left( \aver{B(x, 2\epsilon)} \left| S_k\left( f \one_{B(x, 3 \epsilon)}  \right) \rg^2 dz  \right)^{1/2} \lesssim 2^{-k \rho} \left(  \aver{B(x, 3\epsilon)} \lft f \rg^{p_0} dz \right)^{1/{p_0}},
\end{equation*}
where we used the fact that $\epsilon 2^k \leq 1$. Moreover, $\rho$ is strictly greater than
\begin{equation}
\label{eq:formula-rho}
\tilde{\rho}:=\frac{\frac{1}{p_0}-\frac{1}{2}}{\frac{1}{p}-\frac{1}{2}} \cdot \left( n \left( \frac{1}{p}-\frac{1}{2} \right)-\frac{1}{2}   \right)=\frac{\frac{1}{p_0}-\frac{1}{2}}{\frac{1}{p}-\frac{1}{2}} \cdot \delta(p).
\end{equation}
In particular, if we want $\tilde{\rho}$ as small as possible, we recover the critical index $\ds \rho_n\left( p_0\right)$ from \eqref{formula_rho_n}. This ends the proof of Proposition \ref{non-Heisenberg}.
\end{proof}

\section{Boundedness of new maximal Bochner-Riesz operators: proof of Proposition \ref{propo:maxBR}}\label{sec:maximal_op}

We begin with a few remarks on the statement of Proposition \ref{propo:maxBR}:
\begin{remark}
 \begin{itemize}
 \item the Bochner-Riesz conjecture asserts exactly that $\tilde \delta(p_1)$ defined in \eqref{cond:delta-p_1} should be equal to $\delta(p_1)$
\[
\delta\left( p_1\right)=\max \left( n \lft \frac{1}{p_1}-\frac{1}{2}\rg -\frac{1}{2}, 0 \right).
\]

\item In particular, if $p_0 \geq \frac{2(n+1)}{n+3}$ then $p_1\geq \frac{2n}{n+1}$, and so the Bochner-Riesz conjecture would imply that $\tilde \delta(p_1)=0$. And so, $\ds\bar{\delta}_n(p_0)=\frac{n-1}{2} \left(\frac{1}{p_0}-\frac{1}{2}\right)$.

\item Using the inequality $\ds \tilde{\delta}(p_1) \geq \delta(p_1)$, which corresponds to the necessary condition in the Bochner-Riesz conjecture, we note that $\ds \bar{\delta}_n(p_0) \geq \rho_n(p_0)$.

\end{itemize}
\end{remark}

\begin{proof}[Proof of Proposition \ref{propo:maxBR}]
Using Proposition \ref{non-Heisenberg}, for every $x, y$ and every $\epsilon>0$ with $|x-y|<\epsilon$, we get 
\begin{align*}
  \left( \aver{B\left( y, \epsilon \right)} \vert \ic{B}_\epsilon^\delta (f {\bf 1}_{B(x,3\epsilon)})  \vert^2 dz \right)^{1/2} & \lesssim   \left( \aver{B\left( x, 2\epsilon \right)} \vert \ic{B}_\epsilon^\delta (f {\bf 1}_{B(x,3\epsilon)})  \vert^2 dz \right)^{1/2} \\
  &   \lesssim \sum_{\genfrac{}{}{0pt}{}{k}{2^k \epsilon\leq 1}}  2^{k\delta} \left( \aver{B\left( x, \epsilon \right)} \vert S_k (f {\bf 1}_{B(x,3\epsilon)})  \vert^2 dz \right)^{1/2} \\
  & \lesssim \sum_{\genfrac{}{}{0pt}{}{k}{2^k \epsilon\leq 1}}  2^{k\delta} 2^{-k\rho} \left( \aver{B\left( x, 3\epsilon \right)} \vert f  \vert^{p_0} dz \right)^{1/p_0} \\
  & \lesssim \ic{M}_{p_0}[f](x),
 \end{align*}
 whenever $\ds \delta>\rho > \bar{\delta}_n(p_0) \geq \rho_n \left( p_0 \right)$. Because we can take $\rho$ arbitrarily close to $\rho_n \left( p_0 \right)$, for now, the only viable constraint is that $\ds \delta >\rho_n \left( p_0 \right)$.

Since the $L^{p_0}$ Hardy-Littlewood maximal function $\ic{M}_{p_0}$ is of weak type $(p_0,p_0)$, this term is acceptable and we are reduced to estimating only the second maximal operator:
 \begin{equation*}
\ic{B}^{\delta, **}(f)(x):=\sup_\epsilon \ \sup_{|x-y|<\epsilon} \ \left( \aver{B\left( y, \epsilon \right)} \vert \ic{B}_\epsilon^\delta (f)  \vert^2 dz \right)^{1/2}.
\end{equation*}

 For fixed $x, y\in \rr{R}^n$ and $\epsilon >0$ with $|x-y|<\epsilon$, we write $\ds \ic{B}^{\delta}_\epsilon$ as
\begin{equation*}
\ic{B}^{\delta}_\epsilon(f)=\sum_{\genfrac{}{}{0pt}{}{k\leq 0}{2^k \epsilon \leq 1}}  2^{\delta k} S_k(f).
\end{equation*}
Noting that $\check{s}_k\left( \cdot \right)$ satisfies the decaying estimate
\[
\check{s}_k\left( z\right) \lesssim \frac{2^k}{\left( 1+\lft z\rg  \right)^{\frac{n-1}{2}} \left( 1+2^k \lft z \rg  \right)^N},
\]
we obtain that 
\begin{equation*}
\left \| S_k\left( f \right) \right\|_{L^\infty\left(B\left(y, \epsilon\right)\right)} \lesssim 2^{-k \left( \frac{n-1}{2} \right)^+} \ic{M}\left[ f \right](x).
\end{equation*}

The idea is to interpolate the above estimate with the trivial inequality
\begin{equation*}
\left( \aver{B\left(y, \epsilon \right)} \lft S_k\left( f \right) \rg^{p_1} dz \right)^{1/{p_1}} \lesssim \ic{M}_{p_1}\left[ S_k(f) \right](x),
\end{equation*}
for some $1\leq p_1\leq 2$ (precisely the one defined in \eqref{def:p_1}), in order to get $L^{p_0}$-$L^2$ estimates. Consequently, pick $0 \leq \theta \leq 1$ so that 
\begin{equation}
\label{eq:def-theta-hi}
\frac{1}{2}=\frac{\theta}{\infty}+\frac{1-\theta}{p_1}=\frac{1-\theta}{p_1}.
\end{equation}
Then we obtain that 
\begin{equation*}
\left( \aver{B\left(y, \epsilon \right)} \lft S_k\left( f \right) \rg^{2} dz \right)^{1/{2}} \lesssim 2^{-k \theta\left( \frac{n-1}{2} \right)^+} \cdot \Big( \ic{M}\left[ f \right](x)  \Big)^\theta \cdot \Big(\ic{M}_{p_1}\left[ S_k(f) \right] (x)\Big)^{1-\theta},
\end{equation*}
which further implies the pointwise estimate
\begin{align*}
\ic{B}^{\delta, **}(f)(x) &\lesssim \sup_\epsilon \ \sup_{|x-y|<\epsilon} \  \sum_{\genfrac{}{}{0pt}{}{k\leq 0}{2^k \epsilon \leq 1}} 2^{k \delta} \left( \aver{B\left(y, \epsilon \right)} \lft S_k\left( f \right) \rg^{2} dz \right)^{1/{2}} \\&\lesssim \sup_\epsilon \ \sum_{\genfrac{}{}{0pt}{}{k\leq 0}{2^k \epsilon \leq 1}} 2^{k\left( \delta-  \theta \left( \frac{n-1}{2}\right)^+ \right)} \Big(\ic{M}\left[ f \right](x)   \Big)^\theta \cdot \Big(  \ic{M}_{p_1} \left[ S_k \left( f \right)  \right](x) \Big)^{1-\theta}.
\end{align*}

We aim to estimate $\ds \ic{B}^{\delta, **}$ in some weak $L^{p_0}$ space, where $p_0$ is given by the formula
\[
\frac{1}{p_0}=\theta+\frac{1-\theta}{p_1}=\theta+\frac{1}{2}.
\]

Letting $\ds \sigma:= \delta - \theta \left( \frac{n-1}{2}\right)^+$, H\"{o}lder's inequality for weak $L^p$ spaces implies that 
\begin{equation*}
\left\| \ic{B}^{\delta, **}(f)\right\|_{L^{p_0, \infty}} \lesssim \sum_{k \leq 0} 2^{k\sigma} \| f \|_{L^{1}}^\theta \cdot \| S_k(f)\|_{L^{p_1}}^{1-\theta}.
\end{equation*}
Now we need to use the $L^{p_1}$ boundedness of $S_k$; we note that we have, uniformly in $k \leq 0$, that
\begin{equation}
\label{cond:delta-p_1}
\| S_k(f) \|_{p_1} \lesssim 2^{-k \tilde{\delta}\left( p_1 \right)^+} \|f\|_{p_1},
\end{equation}
where $\tilde{\delta}\left( p_1 \right)$ is the smallest positive number for which \eqref{cond:delta-p_1} is true.

We get that 
\begin{align*}
\left \|  \ic{B}^{\delta, **}\left( f \right) \right\|_{L^{p_0, \infty}} &\lesssim \sum_{k \leq 0} 2^{k \left( \sigma -\tilde{\delta}\left( p_1 \right) \right)} \left \| f \right\|_{L^1}^\theta \left \| f \right\|_{L^{p_1}}^{1-\theta} \\
&\lesssim \left \| f \right\|_{L^1}^\theta \left \| f \right\|_{L^{p_1}}^{1-\theta} ,
\end{align*}
provided $\sigma> \tilde{\delta}(p_1)$.

Now we use restricted type interpolation to deduce our result: for arbitrary functions $|f| \leq {\bf 1}_E$ for some subset $E\subset {\mathbb R}^2$, we obtain
\begin{align*}
 \left\| \ic{B}^{\delta, **}(f) \right\|_{L^{p_0,\infty}}  & \lesssim |E|^{\theta} |E|^{(1-\theta)/p_1}  = |E|^{1/p_0}.
 \end{align*} 
We conclude that $\ic{B}^{\delta,**}$ is of restricted weak type $(p_0,p_0)$ as soon as 
$$\delta> \bar{\delta}_n\left(p_0 \right):=\theta \frac{n-1}{2}+\tilde{\delta}\left( p_1 \right).$$
For such fixed $\delta$, there is enough room to modify a bit the exponent $p_0$ ( preserving still the above condition), so we get different restricted weak type estimates for exponents around $p_0$ and by interpolation we obtain a strong $L^{p_0}$-boundedness.

Overall, the constraint on $\delta$ is 
\begin{equation*}
\delta> \max \left\lbrace n \left( \frac{1}{p_0}-\frac{1}{2}  \right)-\frac{1}{2},  \frac{n-1}{2} \left( \frac{1}{p_0}-\frac{1}{2}  \right)+\tilde{\delta}\left( p_1 \right)   \right\rbrace.
\end{equation*}

\end{proof}

\section{A close examination of the two-dimensional case}\label{sec:dim2}

In this section, we restrict our attention to the case $n=2$, where the Bochner-Riesz conjecture is fully solved. Moreover, we will focus on the range of Lebesgue exponents $(6/5,6)$, which corresponds to $\delta>1/6$ for the $L^p$-boundedness of $\ic{B}^\delta$. We aim to explain how we can play with Lebesgue exponents to obtain other estimates, more symmetrical around the exponent $2$.

We have seen that for $\delta>\frac{1}{6}$, we have sufficiently localized information of the type $L^{6/5}$-$L^2$, which allows us to obtain the following sparse control: for arbitrary smooth and compactly supported functions $f,g$ there exists a sparse collection $\ic{Q}$ such that
$$
\lft \langle \ic{B}^\delta f, g \rangle \rg \lesssim \sum_{Q \in \ic{Q}} \left( \aver{6Q} \lft f\rg^{6/5}dx\right)^{5/6} \cdot \left( \aver{6Q} \lft g\rg^{2}dx\right)^{1/{2}} \cdot \lft Q \rg.
$$
 
The Bochner-Riesz operator being a Fourier multiplier is self-adjoint and so it could seem more natural to look for a range of exponents that is symmetric around $2$. Indeed, we are going to prove the following:

\begin{theorem}
\label{thm:general-sparse-bis}
Let $p_0, q_0 \in [6/5,6]$ be any exponent such that $p_0<q_0$ and
$$ \frac{1}{p_0}-\frac{1}{q_0}\leq \frac{1}{3}.$$
Then for any $\delta > 1/6$, and any $f, g$ compactly supported, there exists a sparse collection $\calS$ (depending on $f,g$) with 
\begin{align}\label{eq:bilin3}
	\left| \int \ic B^\delta(f) \cdot g\, dx \right| 
	\leq C \sum_{P \in \calS}  \left(\aver{6P} |f|^{p_0}\, dx\right)^{1/p_0} \left(\aver{6P} |g|^{q_0'}\, dx\right)^{1/q_0'} |P| .
\end{align}
\end{theorem}
\begin{proof}
The idea is that the gain of integrability that we have to use (from $L^{p_0}$ to $L^{q_0}$) is lower than the one we have proved since:
$$ \frac{1}{p_0}-\frac{1}{q_0}\leq \frac{1}{3}= \frac{1}{2}- \frac{5}{6}.$$

We have first to check that we can prove such $L^{p_0}$-$L^{q_0}$ version of Propositions \ref{prop-Heisenberg}, \ref{non-Heisenberg} and \ref{propo:maxBR}. Here, we will preserve the notation  appearing in these propositions.

\medskip

For Proposition \ref{prop-Heisenberg}, what we have proved is that
$$ \| {\bf 1}_{\widetilde{B_r}} S_k {\bf 1}_{B_r} \|_{L^{6/5}\to L^2} \lesssim r^{-2/3} 2^{-k\rho} \left(1+\frac{d(B_r,\widetilde{B_r})}{r}\right)^{-M}$$
for any balls $B_r,\widetilde{B_r}$ of radius $r$ with $d(B_r,\widetilde{B_r})\geq r.$
Since $S_k$ is self-adjoint, we have also 
$$ \| {\bf 1}_{\widetilde{B_r}} S_k {\bf 1}_{B_r} \|_{L^{2}\to L^6} \lesssim r^{-2/3} 2^{-k\rho} \left(1+\frac{d(B_r,\widetilde{B_r})}{r}\right)^{-M}.$$
Interpolating these two estimates, we can prove
$$ \| {\bf 1}_{\widetilde{B_r}} S_k {\bf 1}_{B_r} \|_{L^{p_0}\to L^{q_0}} \lesssim r^{-2\left(\frac{1}{p_0}-\frac{1}{q_0}\right)} 2^{-k\rho} \left(1+\frac{d(B_r,\widetilde{B_r})}{r}\right)^{-M},$$
which then by summing over a covering of balls, implies a $L^{p_0}$-$L^{q_0}$ version of Proposition \ref{prop-Heisenberg}.

\medskip

For Proposition \ref{non-Heisenberg}, which corresponds to the diagonal part, we have proved that 
$$ \| {\bf 1}_{B_r} S_k {\bf 1}_{B_r} \|_{L^{6/5}\to L^2} \lesssim  2^{k/2} $$
for any balls $B_r$ of radius $r$. By duality and then interpolation, we may also obtain that 
$$ \| {\bf 1}_{B_r} S_k {\bf 1}_{B_r} \|_{L^{p_0}\to L^{q_0}} \lesssim 2^{k/2} r^{-2\left(\frac{1}{p_0}-\frac{1}{q_0}-\frac{1}{3}\right)}.$$
By covering $3B_r$ with balls of radius $r$, it yields that 
$$ \left( \aver{B_r} \lft S_k(f \cdot \one_{B_{3r}}) \rg^{q_0} dx \right)^{1/{q_0}}  \lesssim 2^{-k \rho} \left( \aver{B_{3r}} \lft f \rg^{p_0} dx
 \right)^{1/{p_0}},$$
which is a $L^{p_0}$-$L^{q_0}$ version of Proposition \ref{non-Heisenberg}.

\medskip

For Proposition \ref{propo:maxBR}, we now study a $L^{p_0}$-$L^{q_0}$ version of the maximal operators which are
$$
\ic{B}^{\delta, *}(f)(x):= \sup_{\epsilon>0} \ \sup_{|x-y|<\epsilon}\ \left( \aver{B\left(y, \epsilon\right)} \lft \ic{B}^\delta_\epsilon \left( f \cdot \one_{B\left( x, 3 \epsilon \right)^c}  \right)(z) \rg^{q_0} dz  \right)^{1/q_0}
$$
and
$$ 
\ic{B}^{\delta, **}(f)(x):= \sup_{\epsilon>0} \ \sup_{|x-y|<\epsilon} \ \left( \aver{B\left(y, \epsilon\right)} \lft \ic{B}^\delta_\epsilon \left( f \right)(z) \rg^{q_0} dz  \right)^{1/q_0}.
$$
Then following the exact same proof of Proposition \ref{propo:maxBR}, we obtain that these two maximal operators are of weak-type $(p_0,p_0)$ as soon as $\delta>1/6$ (because the exponent $p_1$, used in the proof, will belong to the range $[4/3,4]$ for which we know that $S_k$ are uniformly bounded in $L^{p_1}$ by any positive power of $2^{-k}$).

\medskip

So we have checked that we can easily obtain a $L^{p_0}$-$L^{q_0}$ version of the main technical propositions. Then we may repeat the selection algorithm explained for Theorem \ref{thm:sparse_coll} and this concludes the proof.
\end{proof}

On the other hand, we can obtain estimates for smaller values of $\delta$, but with other kind of constraints:

\begin{theorem}[Theorem \ref{thm:sparse_coll}  in two dimensions]
Let $1 \leq p_0<2$ and define $\bar{\delta}_2$ as
\begin{equation}
\label{eq:def-d-p_0-critical}
\bar{\delta}_2\left(p_0 \right)=\begin{cases}
& \nu_2\left(  p_0\right), \qquad   \text{  if   }  \frac{6}{5} \leq p_0 \leq 2 \\
&\nu_2\left(  p_0\right)+\frac{1}{1-2 \nu_2\left(  p_0\right)} -\frac{3}{2}, \qquad \text{  if   } 1 \leq p_0 \leq \frac{6}{5},
\end{cases}
\end{equation}
where $\ds \nu_2:=\frac{1}{2} \left(\frac{1}{p_0}-\frac{1}{2}\right)$.
Then for any $\delta> \tilde{\delta}_2(p_0)$, and any functions $f$ and $g$ that are compactly supported, there exists a sparse collection $\ic{S}$ (depending on $f$ and $g$), for which
\[
\lft \langle \ic{B}^\delta(f), g \rangle \rg \leq C \sum_{Q \in \calS}  \left(\aver{6Q} |f|^{p_0}\, dx\right)^{1/p_0} \left(\aver{6Q} |g|^{2}\, dx\right)^{1/2} |Q|.
\]
\end{theorem}

The proof of this result relies on the fact that the Bochner-Riesz conjecture is completely solved in two dimensions, and because of this we can write down an exact expression for $\tilde{\delta}_2$ and $\bar{\delta}_2$:

\begin{proposition} On ${\mathbb R}^2$, let $p_0\in(1,2)$, and $\bar{\delta}_2$ be defined by \eqref{eq:def-d-p_0-critical}. Then for every $\delta>\bar{\delta}_2(p_0)$ the maximal operators $\ic{B}^{\delta,*}$ and $\ic{B}^{\delta,**}$ are of weak-type $(p_0,p_0)$.
\end{proposition}


\section{Quantitative weighted estimates and vector-valued extensions for the Bochner-Riesz}\label{sec:weights}
In this section, we present the weighted estimates implied by our Theorem \ref{thm:sparse_coll}, as well as related vector-valued extensions. To state properly the theorems and clarify their proofs, we recall first a few facts about weights, Muckenhoupt $A_p$ classes and the main properties of the weights therein.
\medskip

For  $1<p<\infty$, a weight $w$ (that is, a non-negative, locally integrable function) is a Muckenhoupt $A_p$ weight if it satisfies the condition
\begin{equation*}
[w]_{A_p}:=\sup_{B}\left(\frac{1}{|B|}\int_{B}w(y)\ dy \right)\left(\frac{1}{|B|}\int_{B}w(y)^{1-p'}\ dy \right)^{p-1}<\infty, 
\end{equation*}
where the supremum is taken over all the balls $B$ of $\mathbb{R}^n$. This constant $[w]_{A_p}$ is known as the characteristic constant of the weight $w$. For the limiting case $p=1$, the class $A_{1}$ is defined to be the set of weights $w$ such that
\begin{align*}
[w]_{A_1}:=\sup_B\bigg(\frac{1}{|B|}\int_B w(x)dx \bigg) \bigg(\esssup_{y\in B} w(y)^{-1} \bigg)<+\infty.
\end{align*}
 One of the main features of Muckenhoupt weights is the reverse H\"older property. More precisely, we say that a weight $w$  satisfies the Reverse H\"older inequality with exponent $s>1$ if there exists a constant C such that for every ball $B$ in $\mathbb{R}^n$
\begin{equation}\label{Rhprop}
\left(\frac{1}{|B|}\int_{B}w(y)^s\ dy \right)^{\frac{1}{s}}\leq C \frac{1}{|B|}\int_{B}w(y)\ dy.
\end{equation}
Since all weights in $A_p$ satisfy \eqref{Rhprop} for a certain exponent $s$, it is possible to describe the $A_p$ weights in a different way. More precisely, we can say that $w$ is in the reverse H\"older class $RH_s$ if it satisfies \eqref{Rhprop} and the constant $C$ defines the characteristic of the weight and is denoted by $[w]_{RH_s}$. This definition can be also extended to $s=\infty$ and the constant $C$ will be exactly
\begin{equation*}
[w]_{RH_\infty}:=\sup_B \bigg(\esssup_{y\in B} w(y)\bigg)\bigg(\frac{1}{|B|}\int_B w(y)dy \bigg) <+\infty.
\end{equation*}
Reverse H\"older and $A_p$ classes are related as follows:

\begin{proposition}\label{prop:Ap-Rh}
Let $q\in[1,\infty]$ and $s>1$. Then $w\in A_q\cap RH_s$ if and only if $w^s\in A_{1+s(q-1)}$. Moreover,
\[[w^s]_{A_{1+s(q-1)}}\leq[w]_{A_q}^s[w]_{RH_s}^s.\]
\end{proposition}

We obtain the following weighted norm estimates for the Bochner-Riesz operator:
\begin{theorem}\label{thm:weightsBR-general}
Let $1<p_0<2$, and $\bar{\delta}_n(p_0)$ as in \eqref{eq:delta_n-bar} . For every $p_0<p<2$, with $\delta>\bar{\delta}_n(p_0)$ and for all weights $w \in A_{\frac{p}{p_0}}\cap RH_{\left( \frac{2}{p}  \right)'}$
\begin{equation}
\label{eq:weight-below2}
\|\ic{B}^\delta\|_{L^p(w) \to L^p(w)}\leq C\left([w]_{A_{\frac{p}{p_0}}}[w]_{RH_{(\frac{2}{p})'}}\right)^\alpha,
\end{equation}
with $\alpha:=max\{1/(p-p_0),1/(2-p)\}$ and $C:=C(\delta,p,p_0,n)$ a constant. Similarly, if $2<p<p_0'$ then for all weights $w\in A_{\frac{p}{2}} \cap RH_{\left( \frac{p_0'}{p}  \right)'}$
\begin{equation}\label{eq:weight-above2}
\|\ic{B}^\delta\|_{L^p(w) \to L^p(w)}\leq C\left([w]_{A_{\frac{p}{2}}}[w]_{RH_{(\frac{p_0'}{2})'}}\right)^\alpha,
\end{equation}
with $\alpha:=max\{1/(p-2),(p_0'-2)/(p_0'-p)\}$. Moreover, we obtain that 

\begin{equation}
\label{eq:weight-L2}
\ic{B}^\delta: L^2(w) \to L^2(w), \quad \text{for all } w  \text{ such that }  w^{\frac{2p_0}{2-p_0}} \in A_2.
\end{equation}
\end{theorem}
\begin{proof}
To prove \eqref{eq:weight-below2} and \eqref{eq:weight-above2} we apply \cite[Proposition 6.4]{weights_beyond_CZ} to the bilinear estimate \eqref{eq:bilin} for the respective range of $p$. By interpolation we obtain \eqref{eq:weight-L2}.

\end{proof}

\begin{remark}\label{re:compareWtQuan}
\rm{A few comments regarding Theorem \ref{thm:weightsBR-general} need to be made:
\begin{itemize}
\item We note that in \cite{KL} some extension in terms of mixed $A_p-A_\infty$ characteristic of the weight (as well as two weights inequalities) have been obtained for the weighted estimates of sparse bilinear form, initially obained in \cite[Proposition 6.4]{weights_beyond_CZ}.  Therefore, this could be applied in our current situation. For example, by combining with Theorem \ref{thm:general-sparse-bis}, \cite{KL} yields the following: in dimension 2, as soon as $\delta>1/6$ then we have
$$ \| \ic{B}^\delta \|_{L^2(w) \to L^2(w)} \lesssim [w^3]_{A_2}^{1/6} [w^3+w^{-3}]_{A_\infty}^{1/2}.$$
\item We believe that the quantitative estimates \eqref{eq:weight-below2} and \eqref{eq:weight-above2} are the best as far as we know because if we track the constants in Theorem \ref{th:knownW} \eqref{item-christ} and we apply extrapolation we do not recover such result for the indicated range of $p$.
\item Following Remark \ref{re:p0delta}, we can rephrase the statement in Theorem \ref{thm:weightsBR-general}, so that the range of exponents $p$, and all the other parameters depend on $\delta$. The details are left to the reader.
 \end{itemize}
}
\end{remark}

Finally, the vector-valued estimate for $\ic B^\delta$ results from Theorem \ref{thm:general-sparse-bis}. Accurately, if we apply \cite[Theorem 4.9]{AuscherMartell} to bilinear form estimate \eqref{eq:bilin3} we obtain the following:

\begin{corollary} \label{cor:vector-valued}
Let $n=2$ and $p,q \in [6/5,6]$ be any exponent such that
\begin{equation} \left|\frac{1}{p}-\frac{1}{q}\right| < \frac{1}{3}. \label{eq:condi} \end{equation}
Then for any $\delta > 1/6$, $\ic{B}^\delta$ admits $L^p(\ell^q)$ estimates.
\end{corollary}

We point out that on the range $p,q \in [6/5,6]$, if $p,q\geq 2$ (and so satisfies \eqref{eq:condi}) then vector-valued estimates could be obtained from Theorem \ref{th:knownW} (as well as if $p,q\leq 2$ by duality). The novelty here is to allow the situations $p<2<q$ or $q<2<p$ with \eqref{eq:condi}, which does not seem to be a direct application of previously known results.

\nocite{SteinBigBook}
\bibliographystyle{plain}
\bibliography{harmonic.bib}

\begin{thebibliography}{10}

\bibitem{AIS}
Kari Astala, Tadeusz Iwaniec, and Eero Saksman.
\newblock Beltrami operators in the plane.
\newblock {\em Duke Math. J.}, 107(1):27--56, 2001.

\bibitem{AuscherMartell}
Pascal Auscher and Jos{\'e}~Mar{\'{\i}}a Martell.
\newblock Weighted norm inequalities, off-diagonal estimates and elliptic
  operators. {I}. {G}eneral operator theory and weights.
\newblock {\em Adv. Math.}, 212(1):225--276, 2007.

\bibitem{weights_beyond_CZ}
Fr{\'e}d{\'e}ric Bernicot, Dorothee Frey, and Stefanie Petermichl.
\newblock Sharp weighted norm estimates beyond {C}alder\'on-{Z}ygmund theory.
\newblock 34 pages, to appear in {\it Anal. \& PDE}, 2016.

\bibitem{Buckley}
Stephen~M. Buckley.
\newblock Estimates for operator norms on weighted spaces and reverse {J}ensen
  inequalities.
\newblock {\em Trans. Amer. Math. Soc.}, 340(1):253--272, 1993.

\bibitem{CarMax}
Anthony Carbery.
\newblock A weighted inequality for the maximal {B}ochner-{R}iesz operator on
  {${\bf R}^2$}.
\newblock {\em Trans. Amer. Math. Soc.}, 287(2):673--680, 1985.

\bibitem{CarSeeger}
Anthony Carbery and Andreas Seeger.
\newblock Weighted inequalities for {B}ochner-{R}iesz means in the plane.
\newblock {\em Q. J. Math.}, 51(2):155--167, 2000.

\bibitem{Carleson-Sjolin}
Lennart Carleson and Per Sj{\"o}lin.
\newblock Oscillatory integrals and a multiplier problem for the disc.
\newblock {\em Studia Math.}, 44:287--299. (errata insert), 1972.
\newblock Collection of articles honoring the completion by Antoni Zygmund of
  50 years of scientific activity, III.

\bibitem{CDL}
Mar\'ia Carro, Javier Duoandikoetxea, and Mar\'ia Lorente.
\newblock Weighted estimates in a limited range with applications to the
  {B}ochner-{R}iesz operators.
\newblock {\em Indiana Univ. Math. J.}, 61(4):1485--1511, 2012.

\bibitem{Christ-weightedBR-aeConv}
Michael Christ.
\newblock On almost everywhere convergence of {B}ochner-{R}iesz means in higher
  dimensions.
\newblock {\em Proc. Amer. Math. Soc.}, 95(1):16--20, 1985.

\bibitem{CiVaStem}
{\'O}scar Ciaurri, Krzysztof Stempak, and Juan~L. Varona.
\newblock Uniform two-weight norm inequalities for {H}ankel transform
  {B}ochner-{R}iesz means of order one.
\newblock {\em Tohoku Math. J. (2)}, 56(3):371--392, 2004.

\bibitem{Conde-Rey}
Jos\'e~M. Conde-Alonso and Guillermo Rey.
\newblock A pointwise estimate for positive dyadic shifts and some
  applications.
\newblock \url{http://arxiv.org/abs/1409.4351}, Sept 2014.
\newblock 19 pages; online.

\bibitem{CorDisc}
A.~C{\'o}rdoba.
\newblock An integral inequality for the disc multiplier.
\newblock {\em Proc. Amer. Math. Soc.}, 92(3):407--408, 1984.

\bibitem{Duo-book}
Javier Duoandikoetxea.
\newblock {\em Fourier analysis, Graduate Studies in Mathematics}, volume~29.
\newblock American Mathematical Society, Providence, RI, Translated and revised
  from the 1995 Spanish original by David Cruz-Uribe, 2001.

\bibitem{DMOS}
Javier Duoandikoetxea, Adela Moyua, Osane Oruetxebarria, and Edurne Seijo.
\newblock Radial {$A_p$} weights with applications to the disc multiplier and
  the {B}ochner-{R}iesz operators.
\newblock {\em Indiana Univ. Math. J.}, 57(3):1261--1281, 2008.

\bibitem{JaviRdF}
Javier Duoandikoetxea and Jos{\'e}~L. Rubio~de Francia.
\newblock Maximal and singular integral operators via {F}ourier transform
  estimates.
\newblock {\em Invent. Math.}, 84(3):541--561, 1986.

\bibitem{ball_multiplier}
Charles Fefferman.
\newblock The multiplier problem for the ball.
\newblock {\em Annals of Mathematics}, pages 330--336, 1971.

\bibitem{Feff-BR-restr}
Charles Fefferman.
\newblock A note on spherical summation multipliers.
\newblock {\em Israel J. Math.}, 15:44--52, 1973.

\bibitem{FKP}
Robert~A. Fefferman, Carlos~E. Kenig, and Jill Pipher.
\newblock The theory of weights and the {D}irichlet problem for elliptic
  equations.
\newblock {\em Ann. of Math. (2)}, 134(1):65--124, 1991.

\bibitem{GrafakosMF}
Loukas Grafakos.
\newblock {\em Modern {F}ourier analysis}, volume 250 of {\em Graduate Texts in
  Mathematics}.
\newblock Springer, New York, 2009.

\bibitem{Herz}
Carl~S. Herz.
\newblock On the mean inversion of {F}ourier and {H}ankel transforms.
\newblock {\em Proc. Nat. Acad. Sci. U. S. A.}, 40:996--999, 1954.

\bibitem{Hytonen}
Tuomas~P. Hyt\"onen.
\newblock The sharp weighted bound for general {C}alder\'on-{Z}ygmund
  operators.
\newblock {\em Ann. of Math. (2)}, 175(3):1473--1506, 2012.

\bibitem{Ap-RH}
Raymond Johnson and Christoph~J. Neugebauer.
\newblock Change of variable results for {$A_p$}- and reverse {H}\"older {${\rm
  RH}_r$}-classes.
\newblock {\em Trans. Amer. Math. Soc.}, 328(2):639--666, 1991.

\bibitem{LaceyA_2}
Michael Lacey.
\newblock {An elementary proof of the $A_2$ Bound}.
\newblock 14 pages, to appear in {\it Israel J. Math.}, 2015.

\bibitem{Lerner2}
Andrei~K. Lerner.
\newblock On an estimate of {C}alder{\'o}n-{Z}ygmund operators by dyadic
  positive operators.
\newblock {\em J. Anal. Math.}, 121:141--161, 2013.

\bibitem{Lerner3}
Andrei~K. Lerner.
\newblock A simple proof of the {$A_2$} conjecture.
\newblock {\em Int. Math. Res. Not. IMRN}, 14:3159--3170, 2013.

\bibitem{LernerNazarov}
Andrei~K. Lerner and Fedor Nazarov.
\newblock Intuitive dyadic calculus: the basics.
\newblock \url{http://arxiv.org/abs/1508.05639}, August 2015.
\newblock online.

\bibitem{KL}
Kangwei Li.
\newblock Two weight inequalities for bilinear forms.
\newblock \url{http://arXiv.org/abs/1511.07250v1}, November 2015.
\newblock 16 pages; online.

\bibitem{CorrigeLiSun}
Kangwei Li and Wenchang Sun.
\newblock Corrigendum to ``{S}harp bound of the maximal {B}ochner-{R}iesz
  operator in weighted {L}ebesgue space'' [{J}. {M}ath.\ {A}nal.\ {A}ppl.\ 395
  (2012) 385--392] [mr2943630].
\newblock {\em J. Math. Anal. Appl.}, 405(2):746, 2013.

\bibitem{LuYan}
Shanzhen Lu and Dunyan Yan.
\newblock {\em {B}ochner {R}iesz Means on {E}uclidean Spaces}.
\newblock World Scientific, 2013.

\bibitem{LPR}
Teresa Luque, Carlos P{\'e}rez, and Ezequiel Rela.
\newblock Optimal exponents in weighted estimates without examples.
\newblock {\em Math. Res. Lett.}, 22(1):183--201, 2015.

\bibitem{PetermichlVolberg}
Stefanie Petermichl and Alexander Volberg.
\newblock Heating of the {A}hlfors-{B}eurling operator: weakly quasiregular
  maps on the plane are quasiregular.
\newblock {\em Duke Math. J.}, 112(2):281--305, 2002.

\bibitem{ShiSun}
Xianliang Shi and Qiyu Sun.
\newblock Weighted norm inequalities for {B}ochner-{R}iesz operators and
  singular integral operators.
\newblock {\em Proc. Amer. Math. Soc.}, 116(3):665--673, 1992.

\bibitem{Stein_problems}
Elias~M. Stein.
\newblock Some problems in harmonic analysis.
\newblock In {\em Harmonic analysis in {E}uclidean spaces ({P}roc. {S}ympos.
  {P}ure {M}ath., {W}illiams {C}oll., {W}illiamstown, {M}ass., 1978), {P}art
  1}, Proc. Sympos. Pure Math., XXXV, Part, pages 3--20. Amer. Math. Soc.,
  Providence, R.I., 1979.

\bibitem{SteinBigBook}
Elias~M. Stein.
\newblock {\em Harmonic analysis: real-variable methods, orthogonality, and
  oscillatory integrals}, volume~43 of {\em Princeton Mathematical Series}.
\newblock Princeton University Press, Princeton, NJ, 1993.
\newblock With the assistance of Timothy S. Murphy, Monographs in Harmonic
  Analysis, III.

\bibitem{Tao_weak-bounds}
Terence Tao.
\newblock Weak-type endpoint bounds for {R}iesz means.
\newblock {\em Proc. Amer. Math. Soc.}, 124(9):2797--2805, 1996.

\bibitem{Tao-BR-impliesRestr}
Terence Tao.
\newblock The {B}ochner-{R}iesz conjecture implies the restriction conjecture.
\newblock {\em Duke Math. J.}, 96(2):363--375, 1999.

\bibitem{Tao_restriction-conjecture}
Terence Tao.
\newblock Some recent progress on the restriction conjecture.
\newblock In {\em Fourier analysis and convexity}, Appl. Numer. Harmon. Anal.,
  pages 217--243. Birkh\"auser Boston, Boston, MA, 2004.

\bibitem{Tomas:tomas-Stein}
Peter~A. Tomas.
\newblock A restriction theorem for the {F}ourier transform.
\newblock {\em Bull. Amer. Math. Soc.}, 81:477--478, 1975.

\bibitem{Vargas}
Ana Vargas.
\newblock Weighted weak type {$(1,1)$} bounds for rough operators.
\newblock {\em J. London Math. Soc. (2)}, 54(2):297--310, 1996.

\end{thebibliography}

\end{document}